\documentclass[12pt]{amsart}
\usepackage[utf8]{inputenc}
\usepackage{amsmath}
\usepackage{amsfonts}
\usepackage{amssymb}
\usepackage{amsthm}
\usepackage{enumerate}
\usepackage[all]{xy}
\usepackage{xfrac}
\usepackage{comment}
\usepackage{mathtools}
\usepackage[left=3.00cm, right=3.00cm, top=2.00cm, bottom=2.00cm]{geometry}

\usepackage{pgfplots}
\pgfplotsset{compat=1.15}
\usepackage{mathrsfs}
\usetikzlibrary{arrows}

\usepackage{xcolor}

\theoremstyle{definition}

\newtheorem{teo}{Theorem}[section]
\newtheorem{defi}[teo]{Definition}
\newtheorem{prop}[teo]{Proposition}
\newtheorem{cor}[teo]{Corollary}
\newtheorem{lema}[teo]{Lemma}

\newtheorem{obs}[teo]{Remark}

\newtheorem*{mthm}{Main Theorem}

\newcommand{\A}{\mathbb{A}}
\newcommand{\C}{\mathbb{C}}

\newcommand{\oc}{\mathcal{O}}

\newcommand{\calo}{\mathcal{O}}
\newcommand{\I}{\mathcal{I}}
\newcommand{\Z}{\mathbb{Z}}
\newcommand{\p}{\mathbb{P}}
\newcommand{\Q}{\mathbb{Q}}
\newcommand{\homc}{\mathcal{H}om}
\newcommand{\extc}{\mathcal{E}xt}
\newcommand{\coh}[2]{H^{#1}(#2)}
\newcommand{\nc}{\mathcal{N}}
\newcommand{\mc}{\mathcal{M}}

\newcommand{\N}{\mathbb{N}}
\newcommand{\hc}{\mathcal{H}}
\newcommand{\E}{\mathcal{E}}

\DeclareMathOperator{\im}{Im}
\DeclareMathOperator{\spec}{Spec}
\DeclareMathOperator{\Hom}{Hom}
\DeclareMathOperator{\supp}{Supp}
\DeclareMathOperator{\quot}{Quot}
\DeclareMathOperator{\rk}{rk}

\DeclareMathOperator{\Ext}{Ext}
\DeclareMathOperator{\ext}{ext}
\DeclareMathOperator{\codim}{codim}
\DeclareMathOperator{\Aut}{Aut}
\DeclareMathOperator{\Hilb}{Hilb}

\DeclareMathOperator{\Sym}{Sym}
\DeclareMathOperator{\Pic}{Pic}

\DeclareMathOperator{\GL}{GL}
\DeclareMathOperator{\PGL}{PGL}

\begin{document}

\definecolor{ffffww}{rgb}{1,1,0.4}
\definecolor{qqccqq}{rgb}{0,0.8,0}
\definecolor{ttffqq}{rgb}{0.2,1,0}
\definecolor{ffqqqq}{rgb}{1,0,0}

\title[Moduli spaces of quasi-trivial sheaves]{Moduli spaces of quasi-trivial sheaves}

\author{Douglas Guimarães}
\address{Université de Bourgogne - IMB \\ UFR Sciences et Techniques -- Bâtiment Mirande, 9 Avenue Alain Savary, 21078 Dijon Cedex, France}
\email{douglas-manoel.guimaraes@u-bourgogne.fr}
\author{Marcos Jardim}
\address{Universidade Estadual de Campinas - UNICAMP \\ Instituto de Matemática, Estatística e Computação Científica - IMECC \\ Departamento de Matem\'atica \\
	Rua S\'ergio  Buarque de Holanda, 651\\ 13083-970 Campinas-SP, Brazil}
\email{jardim@ime.unicamp.br}


\maketitle

\begin{abstract}
A torsion-free sheaf $E$ on a projective variety $X$ is called \emph{quasi-trivial} if $E^{\vee\vee}=\oc_{X}^{\oplus r}$. While such sheaves are always $\mu$-semistable, they may not be semistable. We study the Gieseker--Maruyama moduli space $\nc_X(r,n)$ of rank $r$ se\-mi\-sta\-ble quasi-trivial sheaves on $X$ with $E^{\vee\vee}/E$ being a 0-dimensional sheaf of length $n$ via the Quot scheme of points $\quot(\oc_{X}^{\oplus r},n)$. We show that, when $(X,A)$ is a \emph{good} projective variety, then $\nc_X(r,n)$ is empty when $r>n$, while $\nc_X(n,n)$ has no stable points and is isomorphic to the symmetric product $\Sym^n(X)$. Our main result is the construction of an irreducible component of $\nc_X(r,n)$ of dimension $n(d+r-1)-r^2+1$ when $r<n$. Furthermore, if we restrict to $X=\p^3$ this is the only irreducible component when $n\le10$.
\end{abstract}

\tableofcontents

\section{Introduction}

Let $(X,A)$ be a polarized projective variety, let $P$ be a polynomial in $\Q[t]$, and denote by $\mc_X(P)$ the Gieseker--Maruyama moduli space of semistable sheaves on $X$ with Hilbert polynomial $P$. Maruyama proved in \cite{Maruyama} that the space $\mc_X(P)$ is a projective scheme. However, especially when $\dim X\ge3$, the geometry of such a scheme remains largely unknown, despite the efforts of many authors in the past four decades, and questions about connectedness, irreducibility, the number of irreducible components, and so on, remain open even when we restrict to work over projective spaces taking, for example, $X=\p^3$.

For instance, let $X=\p^3$ and denote $\mc_{\p^3}(P)=\mc_{\p^3}(r,c_1,c_2,c_3)$, where $c_1$, $c_2$, and $c_3$ are the first, second, and third Chern classes of $E$, respectively. When $r=1$ and $c_1=0$ (which can always be achieved after twisting by an appropriate line bundle), one gets that $\mc_{\p^3}(1,0,c_2,c_3)$ is isomorphic to the Hilbert scheme $\Hilb^{d,g}(\p^3)$ of 1-dimensional schemes of degree $d=-c_2$ and genus $g=c_3-2c_2$ \cite[Lemma B.5.6]{KPS}, which is known to always be connected \cite{Hartshorne4}. Not much is known in general when $r\ge 2$, though:
\begin{enumerate}
	\item $\mc_{\p^3}(2,c_1,c_2,c_3)$ is irreducible for $c_3=c_2^2-c_2+2$ when $c_1=0$, or $c_3=c_2^2$ when $c_1=-1$, see \cite[Theorem 1.1]{Schmidt} and the references therein;
	\item $\mc_{\p^3}(2,0,2,c_3)$ has 2 irreducible components when $c_3=2$ and it has 3 irreducible components when $c_3=0$ \cite[Section 6]{Jardim1}.
	\item $\mc_{\p^3}(2,-1,2,c_3)$ has 2 irreducible components when $c_3=2$ and it has 4 irreducible components when $c_3=0$ \cite[Main Theorem 3]{Charles}.
\end{enumerate}
Moreover, the moduli spaces in items (2) and (3) are connected. For higher values of $c_2$, one can check that the number of irreducible components of $\mc_{\p^3}(2,c_1,c_2,0)$ grows with $c_2$, see \cite[Proposition 3.6]{Ein}; it is not known whether $\mc_{\p^3}(2,c_1,c_2,c_3)$ is always connected.

The goal of this paper is to explore a somewhat exotic case, namely 
$$\nc_X(r,n):=\mc_X(r\cdot P_X(t)-n) \text{ with } r,n\geq 1,$$
whose points correspond to \emph{quasi-trivial} rank $r$ sheaves, that is, sheaves $E$ on $X$ such that $E^{\vee\vee}=\oc_{\p^3}^{\oplus r}$; this nomenclature is borrowed from Artamkin \cite{A}. The motivation comes from its close relationship, described in the body of the paper, between $\nc_X(r,n)$ and the Hilbert and Quot schemes of points in $X$. Moreover, even though the main focus of this paper is the moduli space of semistable quasi-trivial sheaves, we also provide some results regarding $\mu$-semistable quasi-trivial sheaves.

First, we study $\mu$-semistable sheaves $E$ on $(X,A)$ with $\rk(E)\geq 1$ and vanishing Chern classes, and show that they are always extensions of ideal sheaves of subschemes of $X$ of codimension at least 3, see Lemma \ref{JH in Pic} below. In addition, when $(X,A)$ is a \emph{good} projective variety (cf. Definition \ref{def:good}), we prove that the moduli space of such sheaves is a GIT quotient of a Quot scheme $\quot(\oc_{X}^{\oplus r},u)$, where $u$ is a polynomial of degree less than or equal to $d-3$, see Theorem \ref{thm moduli}.

Here is the main result of this paper. 

\begin{mthm} \hfill
\begin{enumerate}
\item $\nc_X(r,n)$ is empty whenever $r>n$ or $n<0$. 
\item $\nc_X(n,n)$ is isomorphic to $\Sym^n(X)$. 
\item If $r<n$, then $\nc_X(r,n)$ has an irreducible component of dimension $n(d+r-1)-r^2+1$. Moreover, if $X=\p^3$ and $n\leq10$, then $\nc_{\p^3}(r,n)$ is irreducible. 
\end{enumerate}
\end{mthm}

The bound on $n$ comes from the fact that the variety $\mathcal{C}(n)$ of triples of $n\times n$ commuting matrices is known to be irreducible precisely for $n\le 10$; in fact, our conclusion is that $\nc(r,n)$ is irreducible whenever $\mathcal{C}(n)$ is.



The paper is organized as follows. We start by studying reflexive sheaves with vanishing Chern classes in Section \ref{sect reflexive}; we prove a key technical result about the triviality of these sheaves, which is subsequently used in the following sections. In Section \ref{sect mod spaces} we give a criterion that tells when a torsion-free sheaf coming arising as the kernel of an element of $\quot(\oc_{X}^{\oplus r},u)$ is (semi)stability, and explain the relation between sheaf (semi)stability and the GIT-stability with respect to the natural action of $\GL_r$ on $\quot(\oc_{X}^{\oplus r},u)$.

Section \ref{sect quot ext 3} is where we establish items (1) and (2) of the Main Theorem, and we prove many technical results that will be used in sections \ref{irred comp} and \ref{irred comp 2}. In Section \ref{irred comp} we construct an irreducible component of $\nc_X(2,n)$ that we will use this as an induction step to construct an irreducible component for $\nc_X(r,n)$. In the last section, we restrict to $X=\p^3$ and prove that $\nc_{\p^3}(r,n)$ is irreducible for $n<10$.

We work over the complex numbers $\C$. For a torsion-free sheaf $E$, (semi)stable always means Gieseker (semi)stability, while $\mu$-(semi)stability refers to stability in the sense of Mumford--Takemoto. As usual, we denote by the lower capital letters the dimension of the respective cohomology or Ext group: $h^i(F):=\dim H^i(F)$ and $\ext^i(F,G):=\dim \Ext^i(F,G)$.

\subsection*{Acknowledgments}
DG is supported by the CAPES-MathAmSud grant number \linebreak 88881.520221/2020-01; he was also supported by CAPES – Finance Code 001 and by the International Cooperation Program CAPES/COFECUB Foundation at the Université de Bourgogne where many helpful discussions with Daniele Faenzi were made.
MJ is partially supported by the CNPQ grant number 305601/2022-9 and the FAPESP Thematic Project 2018/21391-1.



\section{Semistable reflexive sheaves with vanishing Chern classes} \label{sect reflexive}

Recall that a torsion-free sheaf $E$ on a smooth polarized projective variety $(X,A)$ of dimension $d$ is said to be \emph{$\mu$-(semi)stable} with respect to an ample divisor $A$ if $E$ is torsion-free and
$$ \dfrac{c_1(F)\cdot A^{d-1}}{\rk(F)} < ~(\le)~ \dfrac{c_1(E)\cdot A^{d-1}}{\rk(E)} $$
for every nontrivial subsheaf $F\subset E$. Furthermore, $E$ is \emph{Gieseker (semi)stable} if every nontrivial subsheaf $F\subset E$ satisfies
$$ \dfrac{P_F(t)}{\rk(F)} < ~(\le)~ \dfrac{P_E(t)}{\rk(E)} $$
where
$$ P_E(t):=\chi\big(E\otimes\calo_X(t\cdot A)\big) = \sum_{k=0}^d \alpha_k(E)t^k $$
denotes the Hilbert polynomial of the sheaf $E$ with respect to the divisor $A$. 


For the remainder of this paper, unless otherwise specified, (semi)stability will always refer to Gieseker (semi)stability. Remark that 
$$ \mu\text{-stability} ~~ \Longrightarrow ~~ 
~~ \text{Gieseker stability} ~~ \Longrightarrow ~~ \text{Gieseker semistability} ~~
\Longrightarrow ~~ \mu\text{-semistability}, $$
see \cite[Lemma 1.2.13]{Huy-Lhen}. Moreover, $E$ is $\mu$-(semi)stable if and only if the dual sheaf $E^\vee$ is $\mu$-(semi)stable.

The main characters of the present paper are presented in the following definition, borrowed from Artamkin \cite[p. 452]{A}.

\begin{defi} \label{def:quasi-trivial}
A torsion-free sheaf $E$ on a projective variety $X$ is said to be \textit{quasi-trivial} if $E^{\vee\vee}\simeq\calo_X^{\oplus r}$.
\end{defi}

Quasi-trivial sheaves of rank 1 are just ideal sheaves of subschemes of codimension at least 2. Clearly, quasi-trivial sheaves are properly $\mu$-semistable (ie. $\mu$-semistable but not $\mu$-stable) when $\rk(E)>1$.

Notice that if $E$ is quasi-trivial and $\codim Q_E\ge3$ where $Q_E:=E^{\vee\vee}/E$, then $c_1(E)=c_2(E)=0$. We will now investigate to what extent the converse is true; since we will make frequent use of the following condition, it is worth fixing it in a definition.

\begin{defi} \label{def:v-chern}
A torsion-free sheaf $F$ on a polarized projective variety $(X,A)$ is said to have \textit{vanishing Chern classes} if $c_1(F)\cdot A^{d-1}=c_{2}(F)\cdot A^{d-2}=0$.
\end{defi}

By Grothendieck--Riemann--Roch, $F$ has vanishing Chern classes if and only if
$$ P_{F,2}(t) = r\cdot P_{X,2}(t), $$
where $r=\rk(F)$ and $P_X(t)=\chi\big(\calo_X(t\cdot A)\big)$ is the Hilbert polynomial of the structural sheaf.

A particular case of a result due to Simpson, see \cite[Theorem 2]{Simpson}, establishes that if $F$ is a $\mu$-semistable reflexive sheaf with vanishing Chern classes on $(X,A)$, then $F$ is an extension of $\mu$-stable locally free sheaves with vanishing Chern classes. This has the following interesting consequence.

\begin{lema} \label{JH in Pic}
Let $(X,A)$ be a smooth polarized projective variety such that every $\mu$-stable reflexive sheaf with vanishing Chern classes is a line bundle. If $E$ is a semistable reflexive sheaf with vanishing Chern classes, then its Jordan--H\"older filtration has factors in $\Pic^0(X)$.
\end{lema}

By the Corlette--Simpson correspondence \cite[Corollary 1.3]{Simpson}, the hypothesis of the lemma holds whenever $\pi_1(X)$ is abelian. Examples of such varieties are Fano varieties, rational surfaces, abelian varieties, K3 surfaces, products of the previous ones, and quotients of simply connected varieties by finite abelian groups (e.g., Enriques surfaces). One can also check that the conclusion of the lemma holds when $(X,A)$ is a smooth projective surface with $K_X \cdot A \leq 0$ and satisfying  $\chi(\mathcal{O}_X)=1$.

\begin{proof}
We argue by induction on $\rk(E)=r$.  If $E$ is a rank $2$ semistable reflexive sheaf with $c_1=c_2=0$, then by Simpson's result, $E$ must be an extension of $\mu$-stable locally free sheaves with vanishing Chern classes, that is, we can write $E$ in the following short exact sequence
$$ 0 \longrightarrow L \longrightarrow E \longrightarrow L' \longrightarrow 0,$$
where $L$ and $L'$ are in $\Pic^0(X)$. Thus we can take
$$0 \subset L \subset E$$
as the desired Jordan-H\"older filtration.
	
Now suppose the result is valid for a rank less than $r$ and let $E$ be a semistable reflexive sheaf with $c_1=c_2=0$. Again, by Simpson's result, there exists a filtration of $E$
$$0=G_0 \subset G_1 \subset \cdots \subset G_k \subset E,$$
such that each quotient is a $\mu$-stable locally free sheaf with vanishing Chern classes. Let $F:=G_k$ and $F':=E/G_k$ and consider the following exact sequence
$$0\longrightarrow F \longrightarrow E \longrightarrow F' \longrightarrow 0.$$
	
Since $F'$ is a $\mu$-stable locally free with vanishing Chern classes, the hypothesis implies that $F'\in \Pic^0(X)$. Now we can apply the induction hypothesis to $F$ which has rank $r-1$, and we get a Jordan--H\"older filtration of $F$
$$0=F_0 \subset F_1 \subset \cdots \subset F_l \subset F$$
whose factors are in $\Pic^0(X)$. Finally, we take 
$$0=F_0 \subset F_1 \subset \cdots \subset F_l \subset F \subset E,$$
and this is a Jordan--H\"older filtration for $E$ which satisfies our requirements.
\end{proof}

Since varieties satisfying the hypotheses of Lemma \ref{JH in Pic} will play an important role in the present paper, we introduce the following definition.

\begin{defi} \label{def:good}
A smooth polarized projective variety $(X,A)$ is said to be \textit{good} if $h^1(\calo_X)=0$, $\Pic(X)=\Z$, and every $\mu$-stable reflexive sheaf $F$ with vanishing Chern classes is a line bundle.
\end{defi}

Fano manifolds with Picard rank 1 are perhaps the most relevant examples of good projective varieties. 

We complete this section with the following application of Simpson's result.

\begin{lema}\label{lema rank r}
Let $E$ be a $\mu$-semistable sheaf of rank $r$ on a good polarized projective variety $(X,A)$ of dimension $d\ge3$. If $c_1(E)\cdot A^{d-1}=c_{2}(E)\cdot A^{d-2}=0$, then $E$ is quasi-trivial and $\codim Q_E\ge3$.
\end{lema}

\begin{proof}
Since $E$ is torsion-free, we have $\dim Q_E\leq d-2$ by \cite[II, Corollary of Lemma 1.1.8]{Okonek}. It follows that we can write the Hilbert polynomial of $Q_E$ as $P_{Q_E}(t)=at^{d-2}+q(t)$ for some $a\in\Q_{\ge0}$ and $q \in \Q[t]$ with $\deg(u) \leq d-3$. Since $c_1(E)\cdot A^{d-1}=c_{2}(E)\cdot A^{d-2}=0$, it follows that the Hilbert polynomial of $E$ is of the form $P_E(t)=r\cdot P_X(t) + v(t)$ where $v \in \Q[t]$ with $\deg(v) \leq d-3$ and $r:=\rk(E)$. 
Hence, we can use the standard exact sequence
\begin{equation} \label{eq:std-sqc}
0 \longrightarrow E \longrightarrow E^{\vee\vee} \stackrel{\varphi_E}{\longrightarrow} Q_E \longrightarrow 0
\end{equation}
to note that 
$$\begin{array}{lcl}
P_{E^{\vee\vee}}(t) & = & P_E(t)+P_{Q_E}(t) \\
& = & r\cdot P_X(t)+at^{d-2}+q(t)+v(t).
\end{array}$$
A simple calculation with Grothendieck--Riemann--Roch yields 
$$ a = -\dfrac{1}{(d-2)!}c_2(E^{\vee\vee})\cdot A^{d-2} $$
since $c_1(E^{\vee\vee})\cdot A^{d-1}=c_1(E)\cdot A^{d-1}=0$. Since $E^{\vee\vee}$ is $\mu$-semistable, we can apply the Bogomolov inequality \cite[Theorem 7.1]{Huy-Lhen}, obtaining the following inequality
$$ \Delta(E^{\vee\vee}) = 2r c_2(E^{\vee\vee})\cdot A^{d-2} \ge0 $$
thus $a\le0$. It follows that $a=0$, thus $E^{\vee\vee}$ is a $\mu$-semistable reflexive sheaf of rank $r$ on $X$ with $c_1(E^{\vee\vee})\cdot A^{d-1}=c_2(E^{\vee\vee})\cdot A^{d-2}=0$. According to Simpson's result, $E^{\vee\vee}$ must be a successive extension of copies of $\calo_X$, since $\Pic^0(X)=\{\calo_X\}$. But the hypothesis $h^1(\calo_X)=0$ implies that extensions of $\calo_X$ by itself are trivial, we conclude that $E^{\vee\vee}=\calo_X^{\oplus r}$. Moreover, $P_{Q_E}(t)=q(t)$ so $\deg(P_{Q_E})\le d-3$ and $\codim Q_E\ge3$.
\end{proof}



\section{Moduli spaces of semistable quasi-trivial sheaves} \label{sect mod spaces}

Let $(X,A)$ be a smooth polarized variety of dimension $d$, and let $V_r$ denote a complex vector space of dimension $r$.

If $E$ is a quasi-trivial sheaf of rank $r$ on $X$, then the standard exact sequence \eqref{eq:std-sqc} provides an element $(\varphi_E,Q_E)$ of the $\quot$ scheme $\quot(V_r\otimes\calo_X,u)$, where $u(t)=r\cdot P_X(t)-P_E(t)$ is a rational polynomial with rational coefficients of degree $\le d-2$. Applying the functor $\Hom(\cdot,\calo_X)$ to the exact sequence in display \eqref{eq:std-sqc}, we obtain
\begin{equation} \label{eq:v_r}
V_r\simeq\Hom(E,\calo_X)    
\end{equation}
since $\Hom(Q_E,\calo_X)=0$ (because $Q$ is a torsion sheaf) and
$$ \Ext^1(Q_E,\calo_X)\simeq H^{n-1}(Q\otimes\omega_X)^*=0$$
since $\dim Q_E\le n-2$.

Conversely, consider an element $(\varphi,Q)$ of the $\quot$ scheme $\quot(V_r\otimes\calo_X,u)$ with $u\in \Q[t]$, that is, $V_r\otimes\calo_X\twoheadrightarrow Q$ is an epimorphism onto a sheaf $Q$ with Hilbert polynomial $P_Q(t)=u(t)$. Set $E:=\ker\varphi$; we have a short exact sequence
\begin{equation}\label{sequencia fundamental}
0\longrightarrow E \longrightarrow V_r\otimes\calo_X \longrightarrow Q\longrightarrow0.
\end{equation}
Clearly, $E$ is a torsion-free sheaf of rank $r$; note that
$$ P_E(t) = P_{V_r\otimes\calo_X}(t)-P_Q(t) = r\cdot P_X(t)-u(t) .$$

\begin{prop}\label{ida mu stab}
Given $(\varphi,Q)\in\quot(V_r\otimes\calo_X,u)$ with $\deg(u) \leq d-2$, the sheaf $E:=\ker\varphi$ is a quasi-trivial sheaf of rank $r$.
\end{prop}
\begin{proof}
Applying $\homc(-,\calo_X)$ to sequence \eqref{sequencia fundamental} we have
$$ 0\longrightarrow Q^\vee \longrightarrow (V_r\otimes\calo_X)^\vee \longrightarrow E^\vee \longrightarrow \extc^1(Q,\calo_X). $$
Since $\codim(Q) \geq 2$, we have that $Q^\vee=\extc^1(Q,\calo_X)=0$ by \cite[Proposition 1.1.6]{Huy-Lhen}, so that $E^\vee\cong V_r^*\otimes\calo_X$. Since $E^\vee$ is properly $\mu$-semistable, then so is $E$. 
\end{proof}

However, different points in $\quot(V_r\otimes\calo_X,u)$ may lead to isomorphic quasi-trivial sheaves. Indeed, consider the following action of $\GL(V_r)\simeq\Aut(V_r\otimes\calo_X)$ on $\quot(V_r\otimes\calo_X,u)$:
\begin{equation}\label{action}
g\cdot(\varphi,Q) := (\varphi\circ g,Q).
\end{equation}
Note that $g\cdot(\varphi,Q)$ is clearly in $\quot(V_r\otimes\calo_X,u)$ again. Previous considerations lead to the following theorem characterizing the set of isomorphism classes of quasi-trivial sheaves. Note that scalar multiples of the identity $\mathbf{1}\in\GL(V_r)$ act trivially on the Quot scheme.

\begin{teo} \label{thm:bijection}
There is a bijection between the set of isomorphism classes of rank $r$ quasi-trivial sheaves with Hilbert polynomial $P(t)$ on $X$, and the set of orbits
$$ \quot(V_r\otimes\calo_X,u) \big/ \PGL_r  ~~ {\rm where} ~~ u(t):=r\cdot P_X(t)-P(t). $$
\end{teo}

\begin{proof}
Let $(\varphi,Q)$, $(\varphi',Q')$ be points in $\quot(V_r\otimes\calo_X,u)$ lying in the same $\GL_r$-orbit, that is, there is some $g\in \GL_r$ such that $(\varphi,Q)=g\cdot (\varphi',Q')=(\varphi'\circ g,Q')$. Hence, there is an isomorphism $f:Q\to Q'$ such that $f\circ \varphi=\varphi'\circ g$. Let $E=\ker\varphi$ and $E'=\ker\varphi'$ be the corresponding $\mu$-semistable sheaves given by Proposition \ref{ida mu stab}. Thus we can consider the following diagram
\begin{equation}\label{diag orbits}
\begin{split} \xymatrix{
	0 \ar[r]& E \ar[r] \ar[d]_{h} & V_r\otimes\calo_X \ar[r]^{\varphi}\ar[d]^{g} & Q \ar[r] \ar[d]^{f} & 0 \\
	0 \ar[r] & E' \ar[r] & V_r\otimes\calo_X \ar[r]^{\varphi'} & Q' \ar[r] & 0. 
} \end{split}
\end{equation}
Since $g$ and $f$ are isomorphisms, by the snake lemma, follows that $h:E\to E'$ is also an isomorphism, as desired.
	
Now let $E$ and $E'$ two isomorphic quasi-trivial sheaves; let $h:E\to E'$ be an isomorphism. As we pointed out in the first paragraph of this section, we obtain corresponding points $(\varphi,Q)$ and $(\varphi',Q')$ in $\quot(V_r\otimes\calo_X,u)$ where $u(t):=r\cdot P_X(t)-P(t)$. Since $E^{\vee\vee}\simeq(E')^{\vee\vee}\simeq V_r\otimes\calo_X$, we obtain an induced morphism $h^{\vee\vee}:V_r\otimes\calo_X\to V_r\otimes\calo_X$; if this is not a monomorphism, then the kernel of the induced morphism $\ker h^{\vee\vee} \to Q$ would is actually injective, contradicting the fact that $\ker h^{\vee\vee}$ is torsion-free; it follows that $h^{\vee\vee}$ must be an isomorphism.

We then construct a commutative diagram like the one in display \eqref{diag orbits}, with $g=h^{\vee\vee}$. In particular, we get an isomorphism $f:Q\to Q'$ and a commutative diagram 
$$ \xymatrix{
V_r\otimes\calo_X \ar[r]^{\varphi} \ar[rd]_{\varphi'\circ g} & Q \ar[d]^{f}\\
& Q',} $$
that is, $(\varphi,Q)=(\varphi'\circ g,Q')$ in $\quot(V_r\otimes\calo_X,n)$. Therefore $(\varphi,Q)=g\cdot(\varphi',Q')$.
\end{proof}

\begin{obs} \label{rmk d}
In \cite{Cazzaniga Ricolfi} the authors studied sheaves on $\p^d$ framed along a fixed hyperplane $D\subset \p^d$, that is, pairs $(E,\phi)$ consisting of a torsion-free sheaf $E$ on $\p^d$ together with an isomorphism $E|_D\overset{\sim}{\longrightarrow}\oc_D^{\oplus r}$, where $r=\rk E$. They showed that the moduli space $F_{r,n}(\p^d)$ of $D$-framed sheaves on $\p^d$ with Chern character $v_{r,n}:=(r,0,\ldots,0,-n)$ is isomorphic to the quot scheme of points $\quot(\oc_{\mathbb{A}^d}^{\oplus r},n)$ for $m\geq 3$.
	
By \cite[Corollary 1.6]{Cazzaniga Ricolfi}, for every $D$-framed sheaf $(E,\phi)$ with Chern character $v_{r,n}$, we can fit $E$ in a short exact sequence 
$$ 0\to E \to \oc^{\oplus r}_{\p^d} \to Q \to 0, $$
where $Q$ has finite support contained in $\A^d=\p^d\setminus D$. It follows that $E^{\vee\vee}=\calo^{\oplus r}_{\p^d}$ thus $E$ is a quasi-trivial sheaf.

We believe that similar results should also hold for framed sheaves on more general pairs $(X,D)$ of varieties with divisor.
\end{obs}

It is easy to see that there are quasi-trivial sheaves that are not semistable: just consider $\calo_X\oplus\I_Z$, where $\I_Z$ denotes the ideal sheaf of a subscheme $Z\subset X$ of codimension at least $2$. Our next task is to characterize semistable quasi-trivial sheaves in terms of the corresponding point in a Quot scheme.

\begin{lema} \label{lem:subsheaf}
Let $E$ be a quasi-trivial sheaf. Every saturated subsheaf $F\hookrightarrow E$ with $\mu(F)=0$, and every torsion-free quotient $E\twoheadrightarrow G$ with $\mu(G)=0$ is also quasi-trivial.
\end{lema}
\begin{proof}
Let $E$ be a quasi-trivial sheaf, and let $F\hookrightarrow E$ be a saturated subsheaf with $\mu(F)=0$ so that $G:=E/F$ is torsion-free and $\mu(G)=0$. We then obtain the following commutative diagram
$$\xymatrix{
& 0 \ar[d]& 0\ar[d] & 0\ar[d] & \\
0\ar[r] & F\ar[r]\ar[d] & F^{\vee\vee}\ar[r]\ar[d] & Q_F\ar[d]\ar[r] & 0 \\
0 \ar[r] & E\ar[r]\ar[d] & V_r\otimes\calo_X\ar[d] \ar[r]& Q \ar[r]\ar[d] & 0 \\
0\ar[r] & G\ar[r]\ar[d] & G^{\vee\vee}\ar[r]\ar[d] & Q_G\ar[d]\ar[r] & 0 \\
& 0 & 0 & 0 & .
}$$
Since $F^{\vee\vee}$ is a saturated subsheaf of $V_r\otimes\calo_X$ with $\mu(F^{\vee\vee})=\mu(F)=0$, then \cite[Corollary 1.6.11]{Huy-Lhen} implies that $F^{\vee\vee}=V_s\otimes\calo_X$ for some $s\in\{1,\dots,r-1\}$. It then follows that $G^{\vee\vee}\simeq V_{r-s}\otimes\calo_X$. 
\end{proof}

As a first application of the previous lemma, we characterize the Jordan-Holder filtration of quasi-trivial sheaves.

\begin{prop} \label{lem:jh}
If $E$ is a quasi-trivial sheaf of rank $r$, then $E$ admits a filtration whose factors are ideal sheaves $\I_{Z_i}$ of subschemes $Z_i\subset X$, $i=1,\dots,r$ of codimension at least 2. Moreover, $\sum_i P_{\calo_{Z_i}(t)}=r\cdot P_X(t)-P_E(t)$.
\end{prop}
\begin{proof}
We proceed by induction on the rank $r$. When $r=1$, then $E=\I_Z$, so the claim is true. Assuming that it also holds for rank $r\ge2$, we show that it holds for rank $r+1$.

Since $E$ is properly $\mu$-semistable, let $F\hookrightarrow E$ be the maximal $\mu$-destabilizing subsheaf. The quotient sheaf $G:=E/F$ is a $\mu$-stable sheaf with $\mu(G)=0$, thus $G=\I_{Z_{r+1}}$ for some subscheme $Z_{r+1}\subset X$ of codimension at least 2. Notice that $F$ is a quasi-trivial sheaf of rank $r$ so, by induction, $F$ admits a filtration whose factors are $r$ ideal sheaves of subschemes of codimension at least 2. It follows that $E$ also admits such a filtration, with factors being all the factors of $F$ plus $\I_{Z_{r+1}}$.

For the last formula, just note that
$$ P_E(t) = \sum_{i=1}^r P_{\I_{Z_i}}(t) = r\cdot P_X(t) - \sum_{i=1}^r P_{\calo_{Z_i}}(t) $$
\end{proof}

Now let $(\varphi,Q)$ be an element in $\quot(V_r\otimes\calo_X,u)$ and set $E=\ker\varphi$. We can relate the (semi)stability of $E$ with the GIT-stability of $(\varphi,Q)$ with respect to the $\GL_r$-action on $\quot(V_r\otimes\calo_X,u)$ given by \eqref{action}. To do this, we use \cite[Lemma 4.4.5]{Huy-Lhen}: a closed point $(\varphi,Q)$ in $\quot(V_r\otimes\calo_X,u)$ is GIT-(semi)stable if, and only if, for every non-trivial proper linear subspace $V_s\subset V_r$ the induced subsheaf $Q':=\varphi(V_s\otimes\calo_X)\hookrightarrow Q$ satisfies the following inequality:
\begin{equation} \label{GIT sst}
\dfrac{P_{Q'}}{s} > ~~ (\geq)~~  \dfrac{P_Q}{r}.
\end{equation}

\begin{teo} \label{thm moduli}
Let $(\varphi,Q)$ be an element in $\quot(\calo_X^{\oplus r},u)$ and let $E=\ker\varphi$. Then $(\varphi,Q)$ is GIT-(semi)stable with respect to the $\GL_r$ action in display \eqref{action} if, and only if, $E$ is (semi)stable.
\end{teo}
\begin{proof}
Assume that $(\varphi,Q)\in\quot(V_r\otimes\calo_X,u)$ is not GIT-(semi)stable, and let $V_s\subset V_r$ be the destabilizing subspace. The kernel sheaf
$$ F:=\ker\{ V_s\otimes\calo_X \hookrightarrow V_r\otimes\calo_X \stackrel{\varphi}{\twoheadrightarrow} Q \} $$
is a subsheaf of $E$ and, setting $Q'$ as above, we get
$$ \dfrac{P_F(t)}{\rk(F)} = \dfrac{s\cdot P_X(t)-P_{Q'}(t)}{s} =
P_X(t) - \dfrac{P_{Q'}(t)}{s} > ~(\ge)~ P_X(t) - \dfrac{P_{Q}(t)}{r} = \dfrac{P_E(t)}{\rk(E)}; $$
therefore, $E$ is not (semi)stable.

Conversely, assume that $E$ is not (semi)stable, let $F\hookrightarrow E$ be a destabilizing subsheaf; since $E$ is $\mu$-semistable, we may assume that $\mu(F)=0$, thus $F$ is also quasi-trivial by Lemma \ref{lem:subsheaf}. Therefore, we obtain a commutative diagram
\begin{equation} \label{diagram V V'}
\xymatrix{
0 \ar[r] & F \ar[r] \ar[d] & V_s \otimes \calo_X \ar[r] \ar[d] & Q_F \ar[r] \ar[d] & 0 \\
0 \ar[r] & E \ar[r] & V_r \otimes \calo_X \ar[r]^{\varphi} & Q \ar[r] & 0.
} \end{equation}
where $Q_F=\varphi(V_s\otimes\calo_X)$. Since $P_F(t)/\rk(F) >(\ge) P_E(t)/\rk(E)$, we conclude that 
$$ \dfrac{P_{Q'}}{s} < (\leq)~ \dfrac{P_Q}{r} = \dfrac{\rk(F)\cdot P_Q}{\rk(E)}, $$
showing that $(\varphi,Q)$ is not GIT-(semi)stable.
\end{proof}

Let $V_r\otimes\calo_{X\times\quot}\stackrel{\mathbf{\varphi}}{\to}\mathbf{Q}$ be the universal object for the quot scheme $\quot(V_r\otimes\calo_X,u)$. The kernel sheaf $\mathbf{E}$ provides a flat family of quasi-trivial sheaves parameterized by $\quot(V_r\otimes\calo_X,u)$. Restricting this family to the open subset $\quot(V_r\otimes\calo_X,u)^{\rm ss}$ consisting of GIT-semistable elements, we obtain a modular morphism
$$ \tilde\Psi ~:~ \quot(V_r\otimes\calo_X,u)^{\rm ss} \longrightarrow \mc_X(P) $$
where $\mc_X(P)$ denotes the Gieseker--Maruyama moduli space of semistable sheaves on $X$ with fixed Hilbert polynomial $P(t)=r\cdot P_X(t)-u(t)$. By virtue of Theorem \ref{thm:bijection} and Theorem \ref{thm moduli}, $\tilde\Psi$ factors through an injective morphism
$$ \Psi ~:~ \quot(V_r\otimes\calo_X,u) /\!\!/ PGL(V_r) \longrightarrow \mc_X(P) $$
whose image is precisely the subset of quasi-trivial sheaves in $\mc_X(P)$. In fact, when $(X,A)$ is a good polarized variety and $\deg(u)\le n-3$, Lemma \ref{lema rank r} implies that $\Psi$ is bijective. Our next goal is to study this situation more clearly.

\begin{lema}\label{lema dim ext1}
Assume that $h^1(\calo_X)=0$. If $E$ is a stable quasi-trivial sheaf of rank $r$ such that $h^{n-2}(Q_E\otimes\omega_X)=0$, then 
$$ \ext^1(E,E) = \hom(E,Q_E) - r^2 + 1 . $$
\end{lema}
\begin{proof}
Note that $\Ext^1(E,\calo_X)\simeq H^{n-1}(E\otimes\omega_X)^*$ by Serre duality. Therefore, twisting the exact sequence in display \eqref{sequencia fundamental} by $\omega_X$ and taking cohomology, the hypotheses imply that  $H^{n-1}(E\otimes\omega_X)=0$.

Applying the functor $\Hom(E,\cdot)$ to sequence \eqref{sequencia fundamental}, we obtain
$$ 0 \to \Hom(E,E) \to \Hom(E,\calo_X)\otimes V_r \to \Hom(E,Q) \to \Ext^1(E,E) \to 0, $$
providing the desired formula. 
\end{proof}


\section{Semistable sheaves with vanishing Chern classes} \label{sect quot ext 3}

Let $(X,A)$ be a polarized variety of dimension at least 3. We are finally ready to focus on the main character of this paper, namely the Gieseker--Maruyama moduli space
\begin{equation} \label{eq:nx}
\nc_X(r,n):=\mc_X(r\cdot P_X(t)-n) ~\text{ with }~ r,n\ge1 .
\end{equation}
Note that $E$ has rank r and vanishing Chern classes; when $(X,A)$ is good (in the sense of Definition \ref{def:good}), then every $E\in\nc_X(r,n)$ is quasi-trivial, and  $Q_E:=E^{\vee\vee}/E$ is a 0-dimensional sheaf of length $n$. Let $\nc^{\rm st}_X(r,n)$ denote the (open, possibly empty) subset consisting of stable sheaves.

The first observation is that $\nc_X(1,n)$ coincides with the Hilbert scheme $\Hilb^n(X)$ of 0-dimensional subschemes of $X$ with length $n$. We will therefore focus on $r\ge2$, and our initial task is to check whether $\nc_X(r,n)$ is non-empty; here is a first easy observation

\begin{prop} \label{nc empty}
If $(X,A)$ is a good polarized variety, then $\nc_X(r,n)=\emptyset$ for $r>n$.
\end{prop}
\begin{proof}
Since $E\in\nc_X(r,n)$ is a semistable, we must have $\coh{0}{E}=0$. Since $E$ is quasi-trivial, we can take cohomology in the short exact sequence in display \eqref{sequencia fundamental}, obtaining 
$$ 0 \longrightarrow \coh{0}{V_r\otimes\calo_X} \longrightarrow \coh{0}{Q} $$
therefore $r\le n$.
\end{proof}

The case $r=n$ turn out to be quite special as well.

\begin{prop} \label{nc(n,n)}
If $(X,A)$ is a good polarized variety, then 
$\nc_X(n,n)={\rm Sym}^n(X)$, and $\nc^{\rm st}_X(n,n)=\emptyset$ when $n\ge1$.
\end{prop}
\begin{proof}
Let $E$ be a sheaf in $\nc_X(n,n)$; by Lemma \ref{lem:jh}, $E$ admits a $n$-step filtration whose factors are sheaves of ideals $\I_{Z_i}$ where each $Z_i$ is 0-dimensional and whose lengths sum to $n$. None of the $Z_i$'s can be empty, otherwise, $E$ could not be semistable, thus each $Z_i$ must have length equal to 1. This means that $E$ cannot be stable and that it is S-equivalent to $\oplus_i\I_{Z_i}$, which is the graded object associated with the Jordan--H\"older filtration of $E$.
\end{proof}

Our next task will be to show that $\nc_X(r,n)\ne\emptyset$ when $r<n$, but this will come as a consequence of several partial results.


\begin{defi} \label{def:unbalanced}
A quasi-trivial sheaf $E$ on polarized variety $(X,A)$ is called \emph{unbalanced} if it admits an epimorphism $E\twoheadrightarrow\I_q$ to the ideal sheaf of a point $q\in X$, and it does not admit a morphism $\I_q \to E$ such that the composition $E \to \I_q \to E$ is the identity.
\end{defi}

Note that if $E$ is unbalanced, then one can consider the kernel of $E \to I_q$ and the following exact sequence
$$0 \to F \to E \to I_q \to 0;$$
note that $F$ is also quasi-trivial; moreover, the second condition in Definition \ref{def:unbalanced} is equivalent to the non-splitting of this previous short exact sequence.

\begin{prop}\label{computation ext for rank r}
	Let $F$ and $G$ be quasi-trivial sheaves  on $(X,A)$ defined by
	\begin{equation} \label{eq def F}
		0 \to F \to V_r \otimes \oc_X \to Q_F \to 0
	\end{equation}
	and
	\begin{equation} \label{eq def G}
		0 \to G \to V_s \otimes \oc_X \to Q_G \to 0,
	\end{equation}
	where $Q_F$ and $Q_G$ are supported on $0$-dimensional subschemes of $X$.
	If $\supp (Q_F) \cap \supp (Q_G) = \emptyset$, then 
	\begin{itemize}
		\item $\Ext^1(F,G)=\rk(F) \cdot H^1(G)$,
		\item $\Ext^2(F,G)=\rk(G) \cdot H^0(Q_F)$
		\item $\Ext^3(F,G)=0$.
	\end{itemize}
\end{prop}

\begin{proof}
	First, taking cohomologies on the sequences defining $F$ and $G$ we obtain that $H^2(F(t \cdot A))=H^2(G(t \cdot A))=0$ for all $t\in \Z$.	Now, apply $\Hom(Q_F,-)$ to sequence \eqref{eq def G} and we obtain $\Ext^i(Q_F,G)=0$ for $i=0,1,2$ and $\Ext^3(Q_F,G)\cong \rk (G) \cdot H^0(Q_F)$. Apply $\Hom(-,G)$ to sequence \eqref{eq def F} and we get $\Ext^3(F,G)=0$, $\Ext^2(F,G)=\rk G \cdot H^0(Q_G)$ and $\Ext^1(F,G)=\rk (F) \cdot H^1(G)$, as desired. 
\end{proof}

\begin{obs}
	Let $F$ be a torsion-free sheaf on $(X,A)$ defined by sequence \eqref{eq def F}. Note that if $F$ is stable, then $H^0(F)=0$. Therefore, taking cohomologies on sequence \eqref{eq def F}, we get that $H^0(Q_F)= V_r \otimes H^0(\oc_X) \oplus H^1(F)$, that is 
	\begin{equation} \label{H^1 de F}
		h^1(F)=h^0(Q_F)-\rk (F).
	\end{equation}
\end{obs}

The next statement guarantees the existence of stable unbalanced sheaves of rank 2.

\begin{lema} \label{lem:nonempty r=2}
For each pair $(q,Z)$ consisting of a point $q\in X$ and a reduced 0-dimensional scheme $Z\subset X$ not containing $q$ with $h^0(\oc_Z)\ge2$, there exists a stable unbalanced sheaf $E\in\Ext^1(\I_q,\I_Z)$ such that for every $p\in Z$ there is an epimorphism $\varepsilon:E\to \I_p$ with $\ker\varepsilon\simeq \I_{Z'}$ where $Z'=\big(Z\setminus\{p\}\big)\cup\{q\}$. In particular, $\nc^{st}_X(2,n)\neq \emptyset$ for every $n\ge3$.
\end{lema}
\begin{proof}
Let $\tilde{Z}=\{p_1,\cdots,p_n\}\subset X$ be a reduced 0-dimensional subscheme, and set 
$$ Z_j:=\{p_1,\cdots,p_n\}\setminus\{p_j\} . $$
Consider the morphism $\varphi:V_2\otimes\calo_X\to\calo_{\tilde{Z}}$ given by the choice of vectors
$$ \alpha_i\in V_2 ~~\text{ for }~~ i=1,\ldots,n
~~\text{ with }~~ \alpha_i \neq \lambda\cdot\alpha_j~~(\lambda\in\C)~~\text{ for all }~~ i\neq j $$
and let $E:=\ker\varphi$. Fix $j\in\{1,\ldots,n\}$ and choose $\beta\in V_2$ such that 
$$ \beta\in\langle\alpha_j\rangle^\perp ~, \text{ and } \beta\notin\langle\alpha_j\rangle^\perp ~~ \text{ whenever } i\neq j, $$

The choice of such vector $\beta$ gives a linear map of $V_1\hookrightarrow V_2$ such that the image of the composition 
$$ V_1\otimes\calo_X \stackrel{\beta}{\longrightarrow} V_2\otimes\calo_X \stackrel{\varphi}{\longrightarrow} \oc_{\tilde{Z}} $$
is precisely $\oc_{Z_j}$, leading to the following commutative diagram:
$$\xymatrix{
& 0 \ar[d] & 0\ar[d] & 0\ar[d] & \\
0 \ar[r]& \I_{Z_j}\ar[d]\ar[r] & V_1\otimes\calo_X \ar[r]\ar[d]^{\beta} & \oc_{Z_j}\ar[r]\ar[d] & 0 \\
0 \ar[r]& E\ar[r]\ar[d] & V_2\otimes\calo_X \ar[r]^{\varphi}\ar[d]&	\oc_{\tilde{Z}}\ar[r]\ar[d] & 0 \\
0 \ar[r]& \I_{p_j}\ar[r]\ar[d] & W_1\otimes\calo_X \ar[r]\ar[d]& \oc_{p_j} \ar[r]\ar[d]& 0 \\
& 0 & 0 & 0 &			
}$$
That is, $E\in\Ext^1(\I_{p_j},\I_{Z_j})$ for all $j=1,\ldots,n$. In particular, $E$ is unbalanced; we claim that $E$ is stable.

Indeed, if $E$ is not stable, then it must be destabilized by an ideal sheaf $\I_Y$ such that $h^0(\oc_Y)\leq \frac{n}{2}$. Let $p_k\in\tilde{Z}$ such that $p_k\notin Y$ and consider $E$ as an element of $\Ext^1(\I_{p_k},\I_{Z_k})$. Hence we have the following diagram.
$$\xymatrix{
& 0 \ar[d]& 0\ar[d] & & \\
& \I_Y \ar@{=}[r] \ar[d] & \I_Y \ar@{-->}[dr]^{0} \ar[d]& & \\
0 \ar[r]& \I_{Z_k} \ar[r] & E \ar[r]& \I_{p_k}\ar[r] & 0
}$$
The composition $\I_Y\to E \to \I_{p_k}$ is zero by the choice of $p_k\notin Y$. Thus the monomorphism $\I_Y\hookrightarrow E$ must factor through $\I_{Z_k}$; however, there can be no non trivial morphism $\I_Y\to\I_{Z_k}$ since $h^0(\oc_{Z_k})=n-1>\frac{n}{2}=h^0(\oc_Y)$ when $n\ge3$. It follows that $E$ must be stable.
\end{proof}

Next, we prove the non-emptiness of $\nc_X(r,n)$ via induction on the rank. 

\begin{prop} \label{prop:nonempty r>2}
Let $(X,A)$ be a polarized variety. Given positive integers $r<n$, there is $E\in\nc_X(r,n)^{\rm st}$ such that $Q_E$ is the structure sheaf of a reduced 0-dimensional subscheme of $X$. Moreover, when $r\ge2$ every such sheaf is unbalanced.
\end{prop}
\begin{proof}
We argue by induction on $r$. For $r=1$, just let $E=\I_Z$, where $Z$ is a reduced 0-dimensional subscheme; it clearly satisfies the desired properties.

When $r\ge2$, the induction hypothesis guarantees the existence of $F\in \nc_X(r-1,n-1)^{st}$ such that $Q_F$ is the structure sheaf of a reduced 0-dimensional subscheme $Z\subset X$ whenever $n>r$. Let $p\notin Z$, and consider an extension of the form 
$$ 0 \longrightarrow F \longrightarrow E \longrightarrow \I_p \longrightarrow 0. $$
It is easy to see that $E$ is quasi-trivial and $Q_E$ is the structure sheaf of $Z\cup\{p\}$. We show that $\Ext^1(\I_p,F)\ne0$ (so that nontrivial extensions do exist), and that $E$ is stable.

\end{proof}

By Lemma \ref{computation ext for rank r}, we have $\Ext^1(\I_p,F)=\rk(\I_p) \cdot H^1(X,F)= H^1(X,F)$. To see that $H^1(X,F)\neq0$, consider the sequence
$$0 \to F \to V_s \otimes \oc_X \to Q_F \to 0,$$
and taking cohomologies, noting that $H^1(X,\oc_X)=0$, we have $h^1(X,F)=h^0(X,F)+h^0(X,Q_F)-\dim V_s > 0$. Now let $E \in \Ext^1(\I_p,F)$ be a nontrivial extension, that is, $E$ can be written as 

$$0 \to F \to E \to \I_p \to 0$$

By  the induction hypothesis and the construction in the previous proposition, it is easy to see that $E$ is a quasi-trivial sheaf. Now we claim that for every point $q \in \supp (Q_E)=\supp(Q_F) \cup \{p\}$, we have $\hom(E,\I_q)=1$. Indeed, apply the functor $\Hom(-,\I_q)$ to sequence
$$0 \to E \to V_r \otimes \oc_X \to Q_E \to 0,$$
and we get

$$\xymatrix{
0\ar[r] & \Hom(Q_E,\I_q) \ar[r] & \Hom(V_r \otimes \oc_X,\I_q) \ar[r] & \Hom(E,\I_q) \\
\ar[r] & \Ext^1(Q_E,\I_q) \ar[r] & \Ext^1(V_r \otimes \oc_X,\I_q) \ar[r] & \cdots
}$$

Note that $\Hom(V_r \otimes \oc_X,\I_q)=V_r \otimes H^0(X,\I_q)=0$ and $\Ext^1(V_r \otimes \oc_X,\I_q)= V_r \otimes H^1(X,\I_q)=0$. It follows that $$\Hom(E,\I_q)\cong \Ext^1(Q_E,\I_q).$$

Now let $\supp(Q_E)=Z \subseteq X$ and apply the functor $\Hom(-,\I_q)$ to sequence
$$0\to \I_Z \to \oc_X \to \oc_Z \to 0.$$
Using the same idea as before, we conclude that $\Ext^1(Q_E,I_p)\cong \Hom(\I_Z,\I_q)$. Utilizing sequence 
$$0 \to \I_q \to \oc_X \to \oc_p \to 0,$$
one can see that $\Hom(\I_Z,\I_q)\cong H^n(X,\omega_X)\cong \C$, that is, $\hom(E,\I_q)=\hom(\I_Z,\I_q)=1$ for every $q\in \supp(Q_E)$.

Observe that the induction hypothesis and the construction of $F$ again imply that we can write $F$ as 
$$0 \to G \to F \to I_q \to 0$$
for some $q \in \supp(Q_F)$. Moreover, the same argument that we used to show that $\hom(E,I_q)=1$ works for every quasi-trivial sheaf, that is, $\hom(F,I_q)=1$ for every $q \in \supp(Q_F)$. Summing up, we have a morphism $F \to \I_q$ and a morphism $E \to \I_q$ with $\hom(F,I_q)=1$. In this case, the composition $F \hookrightarrow E \to \I_q$ must be equal to the morphism $F \to \I_q$, and we can form the following diagram
$$\xymatrix{
& 0 \ar[d] & 0 \ar[d] & 0 \ar[d] & \\
0 \ar[r] & G \ar[r]\ar[d]& K \ar[r]\ar[d] & \I_p \ar[r]\ar@{=}[d] & 0 \\
0 \ar[r]& F \ar[r]\ar[d] & E \ar[r]\ar[d] & \I_p \ar[r]\ar[d] & 0 \\
0 \ar[r]& \I_q \ar@{=}[r]\ar[d] & \I_q \ar[r]\ar[d] & 0 \\
& 0 & 0. & & 
}$$	
It is not hard to see that $K$ is indeed a quasi-trivial sheaf. So this proves that we can write $E$ as an element of $\Ext^1(F,\I_p)$ for every $p\in \supp(Q_E)$ and $F$ quasi-trivial. We want to use a similar argument as we did in the rank 2 case. For this, we note that we have just proved that $\Hom(E,\I_q)=\Hom(\I_Z,\I_q)$, but note that when $q\notin I_Z$, then $\Hom(\I_Z,I_q)$ is zero. It follows that $\Hom(E,I_q)=0$ whenever $q$ is not in $\supp(Q_E)$.

Finally, we can show that $E$ is stable. Assuming, by contradiction, $E$ is not stable, by Theorem \ref{thm moduli}, there is some quasi-trivial sheaf $G\hookrightarrow E$ with $h^0(Q_G)< \frac{s\cdot n}{r}$. Choose $p \in \supp(Q_E)$ such that $p \notin \supp(Q_G)$ and write $E$ as an element of $\Ext^1(F,\I_p)$. By what we have just proved, the composition $G \hookrightarrow E \to \I_p$ is a morphism in $\Hom(G,I_p)$, therefore, it must be equal to zero because $p\notin \supp(Q_G)$. In this case, we can consider the following diagram
$$\xymatrix{
& 0 \ar[d]& 0 \ar[d]  & & \\
& G \ar[d] \ar@{=}[r] & G \ar[d]\ar@{-->}[rd]^0 & & \\
0 \ar[r] & F \ar[r] & E \ar[r] & \I_p \ar[r] & 0
}$$
But $F$ is stable and $G\hookrightarrow F$ with $p(G) < p(F)$. This is a contradiction, therefore $E$ is stable.

\begin{prop}\label{dim family}
	Let $E \in \nc_X^{\rm st}(r,n)$, let $Q_E$ the $0$-dimensional sheaf of length $n$ and let $\I_Z$ be the corresponding sheaf of ideals, that is, $Q_E=\oc_Z$. Then 
	$$\ext^1(E,E)=\hom(I_Z,\oc_Z)+(r-1)(n-1).$$
\end{prop}
\begin{proof}
	Note that $E$ satisfies the hypothesis of Lemma \ref{lema dim ext1},  then we know that $$\ext^1(E,E)=\hom(E,Q_E)-r^2+1.$$
	Now to compute $\hom(E,Q_E)$ we consider the sequence
	$$0 \to E \to V_r \otimes \oc_X \to Q_E \to 0,$$
	and apply the functor $\Hom(-,Q_E)$ to obtain the following exact sequence:
	$$\begin{array}{cccl}
		0 &\to \Hom(Q_E,Q_E) &\to \Hom(V_r\otimes \oc_X,Q_E) &\to \Hom(E,Q_E) \\
		&\to \Ext^1(Q_E,Q_E) &\to \Ext^1(V_r\otimes \oc_X,Q_E) &\to \cdots.
	\end{array}$$

	Since $\hom(V_r\otimes \oc_X,Q_E)=r \cdot h^0(Q_E)$ and $\ext^1(V_r \otimes \oc_X,Q_E)=r \cdots H^1(Q_E)=0$, we have $\hom(E,Q_E)=rn+\ext^1(Q_E,Q_E)-\hom(Q_E,Q_E)$.
	
	Next, we consider the short exact sequence
	$$0 \to I_Z \to \oc_X \to Q_E \to 0,$$
	where we replace $\oc_Z=Q_E$. Applying $\Hom(-,Q_E)$, we have the following exact sequence
	$$\begin{array}{cccl}
		0 &\to \Hom(Q_E,Q_E) &\to \Hom(oc_X,Q_E) &\to \Hom(\I_Z,Q_E) \\
		&\to \Ext^1(Q_E,Q_E) &\to \Ext^1(\oc_X,Q_E) &\to \cdots.
	\end{array}$$
	Using the same idea as before, we conclude that
	$$\hom(Q_E,Q_E)-n+\hom(\I_Z,Q_E)-\ext^1(Q_E,Q_E)=0,$$
	which implies that
	$$\ext^1(Q_E,Q_E)-\hom(Q_E,Q_E)=\hom(I_Z,Q_E)-n.$$
	
	Now we combine the above equalities:
	$$\begin{array}{lcl}
		\ext^1(E,E) & = & \hom(E,Q_E)-r^2+1 \\
		& = & \hom(\I_Z,Q_E)+rn-n-r^1+1 \\
		& = & \hom(\I_Z,Q_E)+n(r-1)-(r^2-1) \\
		& = & \hom(\I_Z,Q_E)+n(r-1)-(r+1)(r-1) \\
		& = & \hom(\I_Z,\oc_Z)+(n-r-1)(r-1).
	\end{array}$$
\end{proof}

\begin{obs}
	Note that when $Z$ is reduced in the above proposition, then we have $\hom(I_Z,\oc_Z)=nd$, where $d=\dim(X)$.
\end{obs}

By the above proof, note that $\Ext^1(E,Q_E) \cong \Ext^2(Q_E,Q_E) \cong \Ext^1(\I_Z,Q_E) = \Ext^1(\I_Z,\oc_Z)$. By \cite[Proposition 2.2.8]{Huy-Lhen} we can conclude that if the Hilbert scheme of points is smooth, then both the Quot scheme of points and $\nc_X(r,n)$ are smooth as well.


\section{Irreducible component of $\nc_X(2,n)$}\label{irred comp}

Returning to the case $r=2$, Proposition \ref{dim family} indicates that we should construct an irreducible family of stable rank 2 quasi-trivial sheaves on $(X,A)$ of dimension $n(d+1)-3$, where $d=\dim(X)$.

To do this, we will need certain general results regarding relative Ext sheaves. To be precise, let $f: X \to Y$ be a morphism of schemes and let $\mathfrak{M}(X)$ and $\mathfrak{M}(Y)$ be the category of $\oc_X$-modules and $\oc_Y$-modules respectively. Let $F\in \mathfrak{M}(X)$. We define $\extc^p_f(F,-)$ to be the right derived functors of the left exact functor $f_*\homc_{\oc_X}(F,-):\mathfrak{M}(X) \to \mathfrak{M}(Y)$.

Let $f:X \to Y$ a flat projective morphism of noetherian schemes and $F$, $G$ coherent $\oc_X$-modules, flat over $Y$. For every $u: Y' \to Y$ of noetherian schemes, we have the base change morphism
$$\tau^i(u):u^*\extc^i_f(F,G) \to \extc_{p_2}^i(p_1^*F,p_1^*G),$$
where $p_1$ and $p_2$ are the projections in the following diagram
$$\xymatrix{
	X \times_Y Y' \ar[r]^{p_1} \ar[d]_{p_2} & X \ar[d]^{f} \\
	Y'\ar[r]_{u} & Y.
}$$
If $y\in Y$ and $u: \spec k(y) \to Y$ is the respective map, we denote the base change morphism by
$$\tau^i(y): \extc^i_f(F,G) \otimes_Y k(y) \to \Ext^i_{X_y}(F_y,G_y).$$

The following result due to Lange will be used several times; see \cite[Theorem 1.4]{Lange}.

\begin{teo} \label{teo base change}
	Let $y\in Y$ be a point and assume the base change morphism $\tau^i(y): \extc^i_f(F,G) \otimes_Y k(y) \to \Ext^i_{X_y}(F_y,G_y)$ to be surjective. Then 
	\begin{enumerate}[(i)]
		\item there is a neighbourhood $U$ of $y$ such that $\tau^i(y')$ is an isomorphism for all $y'\in U$;
		\item $\tau^{i-1}(y)$ is surjective if and only if $\extc_{f}^i(F,G)$ is locally free in a neighbourhood of $y$.
	\end{enumerate} 
\end{teo}

Our next lemma is the crucial technical fact to be explored in the desired irreducible family of quasi-trivial sheaves.

\begin{lema} \label{lema geral}
	Let $f: X \to Y$ be a projective morphism of noetherian schemes with $Y$ reduced, and let $F$ and $G$ be coherent sheaves on $X$ flat over $Y$. If
	$\Ext^3_{X_y}(F_y,G_y)=0$ and the dimension of $\Ext^i_{X_y}(F_y,G_y)$ is constant for $i=1,2$ for all $y\in Y$, then the base change morphism $\tau^1(y)$ is an isomorphism for every $y\in Y$ and $\extc_f^1(F,G)$ is locally free.
\end{lema}
\begin{proof}
	Since $\Ext^3_{X_y}(F_y,G_y)=0$, the base change morphism is $\tau^i(y)$ is trivially surjective. The first item of Theorem \ref{teo base change} implies that the sheaf $\extc^3_f(F,G)$ is zero and item $(ii)$ shows that $\tau^2(y)$ is surjective. Applying Theorem \ref{teo base change} again for $i=2$, we get that $\tau^2(y)$ is an isomorphism. Moreover, \cite[Lemma 1, p. 51]{Mumford} implies that $\extc^2_f(F,G)$ is local free. Finally, applying Theorem \ref{teo base change} again for $i=1$, we obtain that $\tau^1(y)$ is an isomorphism and, since $\ext^1_{X_y}(F_y,G_y)$ is constant, \cite[Lemma 1, p. 51]{Mumford} implies that  $\extc^1_f(F,G)$ is locally free.
\end{proof}

We are finally in the position to construct the family of quasitrival sheaves we are looking for; let $\hc^i$ the universal sheaf for the Hilbert scheme of $i$ points on $X \times \Hilb^i(X)$. Consider the following diagram
$$\xymatrix{
	& \Hilb^1(X)\times \Hilb^{n-1}(X) & \\
	& X\times \Hilb^1(X)\times \Hilb^{n-1}(X) \ar[ld]^{p_1} \ar[rd]_{p_2} \ar[u]_{f} & \\
	X \times \Hilb^1(X) & & X\times \Hilb^{n-1}(X),
}$$

Let
$$ U:=\{(q,Z) \in \Hilb^1(X)\times \Hilb^{n-1}(X) \mid q\notin Z ~,~ Z=Z_{\rm red}\} $$ 
and set $X_U:=X \times U$ with $\pi :X_U \to U$ being the canonical projection; note that $U$ is open in $\Hilb^1(X) \times \Hilb^{n-1}(X)$. Define $\E^i:=\extc^i_{\pi}(p_1^*\hc^1,p_2^*\hc^{n-1})$. Let $(q,Z)$ be a point in $U$ and let $\tau^i(q,Z)$ be the corresponding base change morphism, that is,
$$ \tau^i(q,Z): \extc^i_\pi(p_1^*\hc^1,p_2^*\hc^{n-1}) \otimes k(q,Z) \to \Ext^i_{X}(\I_q,\I_Z). $$
Proposition \ref{computation ext for rank r} and Lemma \ref{lema geral} shows us that $\E^1$ is locally free on $U$. Note that, by Theorem \ref{teo base change}, $\tau^0(q,Z)$ is an isomorphism, that is, for every $(q,z)\in U$ we have
$$\tau^0(q,Z): \extc^0_\pi(p_1^*\hc^1,p_2^*\hc^{n-1}) \otimes k(q,Z) \stackrel{\sim}{\to} \Hom_{X}(\I_q,\I_Z).$$
However, $\Hom_{X}(\I_q,\I_Z)=0$ for every pair $(q,Z)\in U$, which implies that
$$ \extc^0_\pi(p_1^*\hc^1,p_2^*\hc^{n-1})=0 $$
as well. In this case, by \cite[Corollary 4.5]{Lange}, there is an universal extension $\hc$ on $X \times H$ with $H:=\p((\E^1)^\vee)$ such that for every $h\in H$, the restriction $E_h:=\hc|_{X\times\{h\}}$ is a non-split extension of two sheaf of ideals of $0$-dimensional subschemes of $X$ of the following form
\begin{equation} \label{iz-e-iy}
	0 \to \I_Z \to E \to \I_q \to 0.
\end{equation}
In other words, every member $E_{h}$ of the family $\hc$ satisfies the exact sequence in display \eqref{iz-e-iy}, and therefore is an unbalanced sheaf; since stability is an open condition, Proposition \ref{lem:nonempty r=2} guarantees that there is an open subset $H'\subset H$ whose projection $H'\to U$ is surjective and  such that $E_{h}$ is stable for every $h\in H'$.
Therefore, $\hc|_{H'}$ is a family of stable rank 2 quasi-trivial sheaves parameterized by the scheme $H'$, whose dimension can be easily computed as follows
$$ \dim H' = \dim U + \ext^1(\I_q,\I_Z)-1 = d + d(n-1) + n-3 = n(d+1)-3. $$
Moreover, Proposition \ref{dim family} yields $\ext^1(E,E)=n(d+1)-3$.

\begin{teo} \label{teo component}
	For every $n\ge3$, $\nc_X(2,n)$ contains an irreducible component of dimension $n(d+1)-3$.
\end{teo}
\begin{proof}
	Now $\nc_X(2,n)$ is a coarse moduli space, so our family $\hc$ on $X \times H'$ gives us a modular morphism $\Psi: H' \to \nc_X(2,n)$ whose image is precisely the subset of stable unbalanced sheaves. However, as we have seen in Proposition \ref{lem:nonempty r=2}, the representation of an unbalanced sheaf as an extension of the ideal sheaf of a point by the ideal sheaf of a 0-dimensional scheme is not unique, meaning that the morphism $\Psi$ is not injective. Nonetheless, we argue that it is a finite map.
	
	Indeed, note that the Lemma \ref{prop:nonempty r>2} shows that an unbalanced sheaf can be represented as an extension of an ideal sheaf of a point by an ideal sheaf of a reduced 0-dimensional scheme in at most $n$ different ways. In other words, if $E\in\im\Phi\subset\nc_X(2,n)$, then $\Phi^{-1}(E)$ consists of at most $n$ different points.  
	
	This means that the dimension of the image of $\Psi$ in $\nc_X(2,n)$ is equal to the dimension of $H$. Since every $E\in\im\Psi$ satisfies
	$$ \dim T_E \nc_X(2,n) = \ext^1(E,E) = n(d+1)-3 = \dim \im\Psi $$
	we conclude that the closure of $\im\Psi$ within $\nc_X(2,n)$ is an irreducible component of $\nc_X(2,n)$, as desired.
\end{proof}


\section{An irreducible component of $\nc_X(r,n)$}\label{irred comp 2}

In this section, we will construct an irreducible component of $\nc_X^{\rm st}(r,n)$ whenever $r<n$. We will use the results from the first sections, and the previous section as an induction step.

We will argue by induction. The case $r=2$ is already done, so suppose we have a family $\hc_{r,n}$ that gives us an irreducible component $H_{r,n}$ of $\nc_X(r,n)$ with dimension $n(d+r-1)-r^2+1$, that is, $\hc_{r,n}$ is sheaf on $X \times H_{r,n}$ such that for every $h\in H_{r,n}$, $\hc_{r,n}|_h \in \nc_X(r,n)^{st}$.

Consider the following diagram
$$\xymatrix{
	& \Hilb^{1}(X)\times H_{r-1,n-1} & \\
	& X\times \Hilb^{1}(X) \times H_{r-1,n-1} \ar[ld]^{p_1} \ar[rd]_{p_2} \ar[u]_{f} & \\
	X \times \Hilb^1(X) & & X\times H_{r-1,n-1},
}$$
where $H_{r-1,n-1}$ denotes the irreducible component given by the induction hypothesis on $\nc_X(r-1,n-1)$. 

Define
$$U:=\{(y,F) \in \Hilb^1 \times H_{r-1,n-1} \mid y \notin \supp Q_F,~~ \supp Q_F=(\supp Q_F)_{\rm red}\}$$
as an open subset of $\Hilb^1 \times H_{r-1,n-1}$
and let $\mathbb{X}:=X \times Y$ with $\pi: \mathbb{X} \to U$ the projection.

Let $\E^i:=\extc^i_{\pi}(p_1^* \hc^1 , p_2^* \hc_{r-1,n-1})$. By Proposition \ref{computation ext for rank r} and Lemma \ref{lema geral}, $\E^1$ is a locally free on $\Hilb^1 \times H_{r-1,n-1}$ sheaf whose fibres over a point $(y,F) \in X \times \nc_X(r-1,n-1)^{st}$ is $\Ext^1(\I_y,F)$ for every $y\notin \supp Q_F$.

Note that $\Hom_{X}(\I_p,F)=0$ for every $F\in\nc_X(r-1,n-1)^{st}$. So, by Theorem \ref{teo base change}, $\E^0=0$ and \cite[Corollary 4.2]{Lange} implies that there is an universal extension $\hc$ on $X \times H$ with $H=\p((\E^1)^*)$ such that for every $h\in H$, the restriction $\hc|_h$ is a nonsplit extension the following form
\begin{equation} \label{F - Iy}
	0 \to F \to E \to \I_y \to 0.
\end{equation}

In other words, every member $E_{h}$ of the family $\hc$ satisfies the exact sequence in display \eqref{F - Iy}, and therefore is an unbalanced sheaf; since stability is an open condition, Proposition \ref{prop:nonempty r>2} guarantees that there is an open subset $H'\subset H$ whose projection $H'\to U$ is surjective and such that $E_{h}$ is stable for every $h\in H'$.
Therefore, $\hc|_{H'}$ is a family of stable rank $r$ quasi-trivial sheaves parameterized by the scheme $H'$, whose dimension can be computed as follows
$$\begin{array}{lcl}
	\dim H' & = & \dim U + \ext^1(\I_p,F)-1  \\
	& = & d + (n-1)(d+r-1-1)-(r-1)^2+1 + (n-1-(r-1)) -1 \\
	& = & n(d+r-1)-r^2+1.
\end{array}$$
Note that Proposition \ref{dim family} implies $\ext^1(E,E)=n(d+r-1)-r^2+1$.

\begin{teo} \label{irred comp r}
	Let $r<n$ be positive integers. $\nc_X(r,n)$ contains an irreducible component of dimension $n(d+r-1)-r^2+1$.
\end{teo}
\begin{proof}
	$\nc_X(r,n)$ is a coarse moduli space, so our family $\hc$ on $X \times H'$ gives us a modular morphism $\Psi: H' \to \nc_X(r,n)$ whose image is precisely the subset of stable unbalanced sheaves. However, as we have seen in Proposition \ref{prop:nonempty r>2}, the representation of an unbalanced sheaf as an extension of the ideal sheaf of a point by a quasi-trivial sheaf supported in a 0-dimensional scheme is not unique, meaning that the morphism $\Psi$ is not injective. Nonetheless, we argue that it is a finite map.
	
	Indeed, note that the proof of Proposition \ref{prop:nonempty r>2} shows that an unbalanced sheaf can be represented as an extension as in \eqref{F - Iy} in at most $n$ different ways. In other words, if $E\in\im\Phi\subset\nc_X(r,n)$, then $\Phi^{-1}(E)$ consists of at most $n$ different points.  
	
	This means that the dimension of the image of $\Psi$ in $\nc_X(r,n)$ is equal to the dimension of $H$. Since every $E\in\im\Psi$ satisfies
	$$ \dim T_E \nc_X(r,n) = \ext^1(E,E) = n(d+r-1)-r^2+1 = \dim \im\Psi $$
	we conclude that the closure of $\im\Psi$ within $\nc_X(r,n)$ is an irreducible component of $\nc_X(r,n)$, as desired.
\end{proof}


\section{Projective spaces}

In this section we prove that $\nc_{\p^3}(r,n)$ is irreducible for $n\leq 10$. To do this we begin by establishing how the results proved in the previous sections translate to this particular case.
	
	Every sheaf $E\in \nc_{\p^3}(r,n)$ is also $\mu$-semistable, therefore Lemma \ref{lema rank r} implies that $E$ is $\ker \varphi$ for some $(\varphi,Q)\in \quot(\oc_{\p^3}^{\oplus r},n)$ for $n\in \N$. For the converse, Proposition \ref{ida mu stab} says that every $\ker \varphi$ is $\mu$-semistable for $(\varphi,Q)\in \quot(\oc_{\p^3}^{\oplus r},n)$. Lemma \ref{lem:jh} allows us to express $E$ as an extension 
	$$0 \to F \to E \to G \to 0,$$
	with $F\in \quot(\oc_{\p^3}^{\oplus s},k)$ semistable and $G\in \quot(\oc_{\p^3}^{\oplus r-s},n-k)$ stable, for some $k\in\{1,\dots,n-1\}$ and $s\in\{1,\dots,r-1\}$. Moreover, when $r=2$ this implies that $E$ is an extension of a sheaf of ideals
	$$0 \to I_Z \to E \to I_{Z'} \to 0$$
	where $Z$ and $Z'$ are $0$-dimensional subschemes of $\p^3$ with $h^0(\oc_Z)+h^0(\oc_{Z'})=n$.
	
	Theorem \ref{thm moduli} applied to our case translates into the following: $E=\ker \varphi$ for $(\varphi,Q)\in \quot(\oc_{\p^3}^{\oplus r},n)$ is not (semi)stable if, and only if, there is a torsion-free sheaf $F \hookrightarrow E$ with $F=\ker \psi$ for some $(\psi,Q_F)$ in $\quot(\oc_{\p^3}^{\oplus s},k)$ satisfying $0<s<r$ and $$k < ~~ (\leq) ~~ \frac{s\cdot n}{r}.$$
	For the rank $2$ case, this criterion reduces to check whether there is a sheaf of ideals $I_Z \hookrightarrow E$ satisfying 
	$$k < ~~ (\leq) ~~ \frac{n}{2}$$
	where $k=h^0(\oc_Z)$. Moreover $\nc_{\p^3}(r,n)$ is the GIT--quotient of $\quot(\oc_{\p^3}^{\oplus r},n)$ by $\GL_r$.
	
	Our starting point is the fact that the affine Quot scheme $\quot(\oc_{{\A^3}}^{\oplus r},n)$ is irreducible for $n\leq 10$, see \cite{Amar}. In order to see that the same holds for the projective Quot scheme $\quot(\oc_{\p^3}^{\oplus r},n)$, we will use the following technical lemma.
	
	\begin{lema} \label{union irred}
		Let $X$ be the union of two irreducible subschemes $A$ and $B$, and let $C$ be the intersection of $A$ and $B$. Assume $C$ is open in $A$ and $B$. Then $X$ is irreducible
	\end{lema}
	\begin{proof}
		The closure of $A$ in $X$ is the same as the closure of $C$ in $X$. Indeed, a function vanishing on $C$ vanishes on $A$ because $C$ is open in $A$. Also, the closure of $B$ in $X$ is the same as the closure of $C$ in $X$. Since $X=A\cup B$, we conclude that $X$ is the closure of $C$ in $X$. But $C$ is an open subset of an irreducible scheme $A$, thus $X$ is irreducible.
	\end{proof}
	
	As a simple consequence, we have:
	
	\begin{lema}
		If $\quot(\oc_{\A^d}^{\oplus r},n)$ is irreducible, then $\quot(\oc_{\p^d}^{\oplus r},n)$ is irreducible.
	\end{lema}
	\begin{proof}
		Fix coordinates $[x_0:x_1:\ldots:x_m]$ for $\p^d$ and let $H_i=\{x_i = 0\}$ so $A_i:=\p^d\setminus H_i\cong \A^d$. For $(\varphi,Q)\in \quot(\oc_{\p^d}^{\oplus r},n)$, if $\supp(Q)\cap H_i=\emptyset$, then we can restrict $\varphi$ to $A_i$ and we obtain an element of $\quot(\oc_{\A^d}^{\oplus r},n)$. Since we are assuming $\quot(\oc_{\A^d}^{\oplus r},n)$ to be irreducible, we can write $\quot(\oc_{\p^d}^{\oplus r},n)$ as the union of irreducible subschemes such that the two-by-two intersection is open inside those two. Therefore, by induction and the previous lemma, we get the statement.
	\end{proof}
	
	In particular, we conclude that $\quot(\oc_{\p^3}^{\oplus r},n)$ is irreducible for $n\le10$. Theorem \ref{thm moduli} implies that $\nc_{\p^3}(r,n)$ is the GIT-quotient of $\quot(\oc_{\p^3}^{\oplus r},n)$ by $\GL_r$. Thus, whenever $\quot(\oc_{\p^3}^{\oplus r},n)$ is irreducible, $\nc_{\p^3}(r,n)$ is also irreducible. We have therefore established the following claim.
	
	\begin{cor} \label{irred for 10}
		$\nc_{\p^3}(r,n)$ is irreducible for $n\leq 10$.
	\end{cor}
	
	Note that when $n\le 10$, Corollary \ref{irred for 10} tells us that the component constructed in Theorem \ref{teo component} and Theorem \ref{irred comp r} is the only one, and it provides an explicit description of $\nc_{\p^3}(r,n)$ for $n\leq 10$.
	
	\begin{obs} \label{mat comm}
		Let $\mathcal{C}(d,n)$ the variety consisting of $d$-tuples of $n\times n$ commuting matrices. In \cite{JS}, the authors classified the components of $\mathcal{C}(d,n)$ for $n\leq 7$, which gives a classification of the components of $\quot(\oc^{\oplus r}_{\A^d},n)$ for $n\leq 7$, and in \cite[Proposition 6.1]{Amar} the authors proved that the number of irreducible components of the affine quotient scheme $\quot(\oc_{\A^3}^{\oplus r},n)$ is always smaller than or equal to the number of irreducible components of $\mathcal{C}(3,n)$ (regardless of the value of $r$). Therefore, we can conclude that $\nc_{\p^3}(r,n)$ is irreducible whenever $\mathcal{C}(3,n)$ is.
		
		Determining whether $\mathcal{C}(3,n)$ is irreducible is an interesting open problem. Currently, $\mathcal{C}(n)$ is known to be irreducible for $n\le10$ but reducible for $n\ge29$, see \cite{Yongho, HO, Omladic, OCV, S2, S}. Although not yet published, it has also been claimed that $\mathcal{C}(3,11)$ is irreducible, see \cite[p. 271]{H-O}.

  In addition, note that $\nc_{\p^3}(1,n)=\Hilb^n(\p^3)$, which has been shown to be irreducible for $n\le11$, see \cite{HJ} for $n\le10$ and \cite{DJNT} for the case $n=11$.

	\end{obs}

\end{document}